\documentclass[reqno, 12pt]{amsart}
\usepackage{amsmath, amssymb, amsthm, verbatim, bbm} 

\newcommand \N {\mathbb{N}}
\newcommand \R {\mathbb{R}}

\newcommand \Oh {\mathcal{O}}
\newcommand \la {\langle}
\newcommand \ra {\rangle}
\newcommand \wt {\widetilde}
\newcommand \D {\partial}
\newcommand \eps {\varepsilon}
\newcommand \one {\mathbbm{1}}

\DeclareMathOperator \im {Im}

\DeclareMathOperator \supp {supp}

\DeclareMathOperator \WF {WF}
\DeclareMathOperator \Op {Op}
\DeclareMathOperator \Id {Id}

\newtheorem{lem}{Lemma}
\title
[Local smoothing for scattering manifolds]
{Local smoothing for scattering manifolds with hyperbolic trapped sets}

\author[Kiril Datchev]
{Kiril Datchev}
\email{datchev@math.berkeley.edu}
\address{Mathematics Department, University of California \\
Evans Hall, Berkeley, CA 94720, USA}

\begin{document}

\begin{abstract}
We prove a resolvent estimate for the Laplace-Beltrami operator on a scattering manifold with a hyperbolic trapped set, and as a corollary deduce local smoothing. We use a result of Nonnenmacher-Zworski to provide an estimate near the trapped region, a result of Burq and Cardoso-Vodev to provide an estimate near infinity, and the microlocal calculus on scattering manifolds to combine the two.
\end{abstract}

\maketitle
\vspace{-1cm}
\section{Introduction}\label{in}

In this paper, we prove local smoothing and a resolvent estimate for the Laplace-Beltrami operator on a scattering manifold with a hyperbolic trapped set. We exploit the fact that the resolvent estimate of Nonnenmacher-Zworski \cite{nz} in the case where a complex absorbing potential is added does not require an analyticity assumption near infinity, because it does not use the method of complex scaling. To remove the complex absorbing potential from the resolvent estimate, we use a result of Burq \cite{b}, in a more refined form obtained by Cardoso-Vodev \cite{cv}, which estimates a resolvent away from its trapped set. Our setting is the class of scattering manifolds introduced by Melrose in \cite{m}, which we study from the point of view of Vasy-Zworski \cite{vz}, from whom we take an escape function construction and a positive commutator argument.

Our main result, from which the local smoothing follows, is the following resolvent estimate (we defer definitions to section 2):

\textbf{Theorem.} \textit{Let $X$ be a scattering manifold, and let $-\Delta_g$ be the nonnegative Laplace-Beltrami operator on $X$. Let $d$ be the distance function induced by the metric on $X$, and let $y_0$ be a point in the interior of $X$. Suppose that the trapped set of the unit speed geodesic flow, $K \subset T^*X^\circ$, is compact, and that the flow is hyperbolic and with topological pressure which obeys $P(1/2) < 0$ on $K$. Then for any $\beta_0>0$ and $z_0 > 0$, there exists $C \in \R$ such that, for $z \ge z_0$, $0 < \beta \le \beta_0$,}
\begin{equation}\label{theo}\left\|\la d(y,y_0)\ra ^{-\frac 12 - \eps}(-\Delta_g - z -i\beta)^{-1}\la d(y,y_0)\ra ^{-\frac 12 - \eps}\right\|_{L^2(X) \to L^2(X)} \le C \frac{|\log z|}{\sqrt z}.\end{equation}

The hypothesis on the trapped set allows us to apply the results of \cite{nz}. The toplogical pressure is the pressure of the flow on $K$ with respect to the unstable Jacobian, that is to say with respect to the Jacobian of the flow map restricted to the unstable manifold. The bound on the pressure implies that the trapped set is `thin' in a suitable sense. For example, if $\dim X = 2$, it is sufficient to have $\dim K \le 2$. If $X$ is a scattering manifold which has constant negative curvature everywhere outside a sufficiently small neighborhood of infinity, it is sufficient to have $\dim K \le \dim X - 1$. See \cite[Section 3.3]{nz} for more details.

Observe that as a result of the limiting absorption principle (see \cite[Proposition 14]{m}), the limit $\beta \to 0$ of the resolvent exists, and satifies the same estimate:

\[\left\|\la d(y,y_0)\ra ^{-\frac 12 - \eps}(-\Delta_g - z -i0)^{-1}\la d(y,y_0)\ra ^{-\frac 12 - \eps}\right\|_{L^2(X) \to L^2(X)} \le C \frac{|\log z|}{\sqrt z}.\]

We will use a semiclassical approach to this theorem: after a rescaling given by $z = \lambda/ h^2$, the bound $\log z/\sqrt z$ becomes $\log(1/h)/h$. In fact, the crucial result for us will be
\begin{equation}\label{est}
\left\|x^{\frac 12 + \eps}(-h^2\Delta_g - \lambda -i\beta)^{-1}x^{\frac 12 + \eps}\right\|_{L^2(X) \to H_h^2(X)} \le C \frac {\log(1/h)}h,
\end{equation}
for $\lambda > 0$, $\beta_0 > 0$ and $h_0 > 0$ fixed, and for $\beta \in (0,\beta_0)$, $h \in (0,h_0)$. The statement for arbitrary $z_0$ and $\beta_0$ follows from the resolvent identity. Here $x$ is a boundary defining function on $X$, and we will use this in place of $\la d(y,y_0)\ra^{-1}$, which is an example of such a function. Throughout this paper $C$ denotes a constant, which may change from line to line, but which is uniform in $h$ and $\beta$. The same holds for the implicit constants when $\Oh$ notation is used.

From \eqref{est} we will deduce the following local smoothing inequality:
\begin{equation}\label{smooth}\int_0^T \left\|x^{\frac 12 + \eps} e^{it \Delta_g}u\right\|^2_{H^{\frac 12 - \eta}(X)} dt \le C_{\eta,T} \|u\|^2_{L^2(X)},\qquad \eta > 0.\end{equation}

Work by Sj\"olin \cite{s}, Vega \cite{v}, and Constantin-Saut \cite{cs} established this local smoothing estimate with $\eta = 0$ in the case $X = \R^n$. Doi \cite{d} showed that in a wide variety of geometric settings the absence of trapped geodesics is a necessary condition for \eqref{smooth} to hold with $\eta = 0$. Burq \cite{b2} proved \eqref{smooth} for $\eta > 0$ in the case of a trapped set arising from several convex obstacles satisfying certain hyperbolicity assumptions. Christianson \cite{c} proved \eqref{smooth} for $\eta >0$ in the case of a manifold which is Euclidean outside of a compact set, with the same trapping assumptions as in the present paper; the novelty in our result lies in the fact that our assumptions at infinity are weaker.

Andr\'as Vasy has recently suggested a possible direct approach to this result, replacing \cite[(1.5)]{cv} by propagator estimates for the resolvent in the spirit of Section \ref{incoming}.

\section{Preliminaries}

Let $X$ be a compact $C^\infty$ manifold with boundary, and let $x$ be a boundary defining function, that is to say $x \in C^\infty(X;[0,\infty))$ with $x^{-1}(0) = \D X$ and $x = 0 \Rightarrow dx \ne 0$. We use $X^\circ$ to denote the interior of $X$ and say $X$ is a \textit{scattering manifold} if $X^\circ$ is equipped with a metric which takes the following form near $\D X$:
\begin{equation}\label{metric} \frac {dx^2}{x^4} + \frac {h'}{x^2}, \qquad h'|_{\D X} \textrm{ is a metric on }\D X. \end{equation}
Such a metric blows up at $\D X$, and hence cannot be extended to all of $T X$. We accordingly define the scattering tangent bundle, $^{\textrm{sc}}TX$, to be the bundle of vector fields given by $x V_b(X)$, where $V_b(X)$ denotes the space of vector fields tangent to $\D X$, and observe that our metric extends to $^{\textrm{sc}}TX$. The scattering cotangent bundle, $^{\textrm{sc}}T^*X$, is defined to be the dual of $^{\textrm{sc}}TX$. In a collar neighborhood of the boundary, we use coordinates $(x,y)$ on $X$, and $(x,y,\xi,\eta)$ on $T^*X$, and these give rise to `semi-global coordinates',
\[(x,y,\tau,\mu) = (x,y,x^2\xi,x\eta)\]
on $^{\textrm{sc}}T^*X$, coming from the identification
\[\tau \frac {dx} {x^2} + \mu \frac {dy} x = \xi dx + \eta dy.\]
Because the vector fields in $^{\textrm{sc}}TX$ vanish to order $x^2$ in $\D_x$ and to order $x$ in $\D_y$, a corresponding dual growth is permitted in the differential forms of $^{\textrm{sc}}T^*X$.

An important example of this type of manifold is the case where $X$ is a cone near the boundary, i.e. is isometric near infinity to $\D X \times (R,\infty)$ with a metric of the form
\begin{equation}\label{cone} d r^2 + r^2 h', \qquad h'|_{\D X} \textrm{ is a metric on }\D X.  \end{equation}
In this case $r^{-1}$ serves as a boundary defining function in this region, and we see that the above definition agrees with \eqref{metric} under the identification $r^{-1} = x$, as is shown by the computation $dr = d(x^{-1}) = -x^{-2}dx$. We also have
\[\tau \frac {dx} {x^2} + \mu \frac {dy} x = - \tau dr + \mu r dy,\]
which allows us to interpret $-\tau$ as the dual variable to $r$. In the case where $X^\circ = \R^n$, we may take $X$ to be a closed $n$-dimensional hemisphere obtained by radial compactification. The Euclidean metric on $\R^n$ in polar coordinates now takes the form \eqref{cone} near $\D X$, where $h'$ is the round metric on $\mathbb{S}^{n-1} = \D X$.

A function $p \in C^\infty(T^*X^\circ)$ is said to have flow which is \textit{nontrapping} near energy $\lambda$ if there exists $\delta > 0$ such that, for any $\zeta \in T^*X^\circ$ with $\lambda - \delta < p(\zeta) < \lambda + \delta$, we have
\[ \lim_{t \to \infty} x\left[\exp(tH_p)(\zeta)\right] = 0 \qquad \textrm{ and } \qquad \lim_{t \to -\infty} x\left[\exp(tH_p)(\zeta)\right] = 0.\]
Later on we will occasionally use $a(t)$ as shorthand for $a\left[\exp(tH_p)(\zeta)\right]$.

The following lemma gives the fundamental example of a nontrapping flow on a scattering manifold, and is essentially to be found in \cite{m}.

\begin{lem}\label{nontrap}
The symbol of the Laplacian, $|\zeta|^2 = \tau^2 + h'(\mu,\mu)$, has nontrapping flow near $\D X$ at all energies (here $h'$ is a bilinear form which depends on $(x,y)$ and which is evaluated at $(\mu,\mu)$). More precisely, for all $\lambda$ there exists $x_0$ such that if $\zeta_0 \in T^*X^\circ$ satisfies $x(\zeta_0) < x_0$, then either
\[\lim_{t \to \infty}\exp(tH_{|\zeta|^2})(\zeta_0) = 0 \qquad \textrm{or} \qquad \lim_{t \to -\infty}\exp(tH_{|\zeta|^2})(\zeta_0) = 0.\]
\end{lem}

\begin{proof} To see this, we must first study the flow of $|\zeta|^2$ by computing its Hamiltonian vector field, a computation which we adapt from \cite[p. 19]{m}:
\begin{align}\label{ham}
H_{|\zeta|^2} &= \D_\xi |\zeta|^2 \D_x - \D_x |\zeta|^2 \D_\xi + (\D_\eta|\zeta|^2) \cdot \D_y - (\D_y|\zeta|^2)\cdot \D_\eta. \notag
\intertext{We use $\D_\xi = x^2 \D_\tau$, $\D_\eta = x \D_\mu$ and ``$\D_x = \D_x + x^{-1}\mu \cdot \D_\mu + 2\tau x^{-1}\D_\tau$'', where in the last formula the left hand side refers to $(x,y,\xi,\eta)$ coordinates, and the right hand side to $(x,y,\tau,\mu)$ coordinates. This gives}
H_{|\zeta|^2} &= x^2 \D_\tau |\zeta|^2(\D_x + x^{-1} \mu \cdot \D_\mu + 2\tau x^{-1}\D_\tau) \notag\\
&\quad- x \left[\left(x\D_x + \mu \cdot \D_\mu + 2\tau \D_\tau\right)|\zeta|^2\right]\D_\tau + x(\D_\mu|\zeta|^2) \cdot \D_y - x(\D_y|\zeta|^2)\cdot \D_\mu. \notag
\intertext{We cancel the $\D_\tau(|\zeta|^2)2\tau x\D_\tau$ terms, write $H_{h'} = (\D_\mu|\zeta|^2) \cdot \D_y - (\D_y|\zeta|^2)\cdot \D_\mu$, substitute $|\zeta|^2 = \tau^2 + h'(\mu,\mu)$, and use $\mu \cdot \D_\mu h'(\mu,\mu) = 2h'(\mu,\mu)$. Now}
H_{|\zeta|^2} &= 2\tau x^2\D_x + 2\tau x\mu \cdot \D_\mu - (2x h'(\mu,\mu) - x^2 \D_x h'(\mu,\mu))\D_\tau + xH_{h'}.
\end{align}
We now observe from this that, along flowlines of $H_{|\zeta|^2}$, we have $\frac d {dt} x = 2 \tau x^2$ and $\frac d {dt} \tau = -2xh'(\mu,\mu) + x^2 \D_x h'(\mu,\mu)$. This allows us to compute
\[\frac d {dt} x^{-1} \tau =  \tau \frac d {dt} x^{-1} + x^{-1} \frac d {dt} \tau = -2 \tau^2 - 2 h'(\mu,\mu) + x \D_x h'(\mu,\mu).\]
The function $h'(\mu,\mu)$ is smooth up to $\D X$, and hence by taking $x$ small we can make $x \D_x h'(\mu,\mu)$ arbitrarily small. In other words,
\[ \frac d {dt} x^{-1} \tau \le -\tau^2 - h'(\mu,\mu) = -|\zeta|^2, \qquad x \textrm{ sufficiently small.} \]
If we now restrict ourselves to $|\zeta|^2 \in (\lambda - \delta, \lambda + \delta)$, we have
\[\frac d {dt} x^{-1} \tau \le - \lambda + \delta, \]
and as a result
\[ x^{-1}(t)\tau(t) \stackrel{t \to \infty}{\longrightarrow} \, -\infty,\]
provided the trajectory remains in the part of $X$ where these coordinates are defined. If the initial condition has $\tau(0) \le 0$, then by $\frac d {dt} x = 2 \tau x^2$ we see that $x$ is decreasing, and it must approach zero because the conservation of $p$ implies that $\tau$ is bounded. In the case $\tau(0) \ge 0$, the same calculation gives the result as $t \to -\infty$.
\end{proof}

The bundle $^{\textrm{sc}}T^*X$ will be our phase space, and we will use the microlocal calculus developed in \cite{m}, in \cite{wz}, and in \cite{vz}. In particular we use semiclassical Sobolev spaces associated to our scattering metric. We denote by $\|\cdot\|_{L^2(X)}$ the $L^2$ norm on $X$ with respect to this metric, and then put
\[\|u\|_{H^m_h(X)} = \|(\Id - h^2 \Delta_g)^{m/2} u\|_{L^2(X)}.\]

We use the notation $S^{m,l,k}(X)$ to denote the symbol class of functions $a \in C^\infty((0,1) \times T^*X)$ satisfying $h^kx^{-l} \sigma^m a \in L^\infty((0,1) \times T^*X)$, and satisfying the same estimate after the application of any $b$-differential operator on the fiber radial compactification of $^{\textrm{sc}}T^*X$. A $b$-differential operator is an element of the algebra generated by the vector fields tangent to the boundary of the fiber radial compactification of $^{\textrm{sc}}T^*X$, and $\sigma$ is a boundary defining function in the fibers of the fiber radial compactification of $^{\textrm{sc}}T^*X$ (this compactification forms a manifold with corners: see \cite[Section 4]{m}). Symbols with higher $l$ have better decay at spatial infinity, while symbols with lower $m$ have better decay at frequency infinity, i.e. have better smoothing properties. The principal symbol corresponding to a symbol $a \in S^{m,l,k}(X)$ is defined to be the equivalence class of $a$ in $S^{m,l,k}(X)/S^{m-1,l+1,k-1}(X)$.

These symbols can be quantized in the case where $X = \overline {\R^n}$, the radial compactification of $\R^n$ discussed above, using the following quantization formula:
\begin{equation}\label{quant}\Op(a)u(z) = \left(\frac 1 {2\pi h} \right)^n \int e^{i(z-w)\cdot \xi/h}a\left(h,z,\xi\right)u(w)dwd\xi.\end{equation}
A pseudodifferential operator $A \in \Psi^{m,l,k}(\overline{\R^n})$ is one which is obtained by \eqref{quant} from a symbol $a \in S^{m,l,k}(\overline{\R^n})$. This definition can be extended by localization to a general $X$: the necessary invariance under changes of coordinates is proved in \cite[Proposition A.4]{wz}, following \cite{sch}. We quantize a total symbol $a$ by using \eqref{quant} in local coordinates together with a fixed partition of unity, but bear in mind that only the principal symbol is invariantly defined. We say that $A \in \Psi^{m,l,0}$ is \textit{elliptic} on a set $K \subset {}^{\textrm{sc}}T^*X$ if $a$, the principal symbol of $A$, satisfies $|a| \ge c x^l\sigma^{-m}$ on $K$. The map associating a principal symbol to a pseudodifferential operator obeys the standard properties of being commutative to top order, and of taking a commutator to a Poisson braket (see \cite[(2.1)]{vz}). More precisely, given $A \in \Psi^{m,l,k}(X)$ and $B \in \Psi^{m',l',k'}(X)$, we have $[A,B] \in \Psi^{m+m'-1,l+l'+1,k+k'-1}(X)$ with symbol $\frac h i H_ab$.

For $X=\overline{\R^n}$, we define the wavefront set of a pseudodifferential operator $A=\Op(a)$, denoted $\WF_hA$, as follows. For a point $\zeta \in {}^{\textrm{sc}}T^*X^\circ$, we say $\zeta \not\in \WF_hA$ if, in a neighborhood of $\zeta$, $|\D^\alpha a| = \Oh(h^\infty)$ for any multiindex $\alpha$. For a point $\zeta \in \D{}^{\textrm{sc}}\overline{T}^*X$, we say that $\zeta \not\in \WF_hA$ if, in a neighborhood of $\zeta$, $|\D^\alpha a| = x^\infty\sigma^\infty\Oh(h^\infty)$ for any multiindex $\alpha$. Here ${}^{\textrm{sc}}\overline{T}^*X$ is the fiberwise radial compactification of ${}^{\textrm{sc}}T^*X$, and $\sigma$ is again the fiber boundary defining function. That this notion is invariant under coordinate change follows, for example, from \cite[(8.43)]{ez}, and as a result the definition can be extended to any scattering manifold $X$. What will be important for us is that the wavefront set of a product is the intersection of the wavefront sets: i.e. if $A \in \Psi^{m,l,0}(X)$ and $B \in \Psi^{m',l',0}(X)$, then
\begin{equation}\label{comp}\WF_h(AB) \subset \WF_hA \cap \WF_h B.\end{equation}
This containment can be deduced in $\R^n$ from the composition formula \cite[(4.22)]{ez}. The fact that the wavefront set is an invariant feature of a pseudodifferential operator allows the result to be extended to a general scattering manifold $X$.

The wavefront set allows us to define a notion of local invertibility for the region where a pseudodifferential operator is elliptic: Let $A \in \Psi^{m,l,0}(X)$ be elliptic on $K \subset {}^{\textrm{sc}}\overline{T}^*X$. Then there exists $A'\in\Psi^{-m,-l,0}(X)$ such that
\begin{equation}\label{ellip}K \cap \WF_h(A'A - \Id) = \varnothing \quad \textrm{and} \quad K \cap \WF_h(AA' - \Id) = \varnothing.\end{equation}
Indeed, let $a$ be the principal symbol of $A$, and suppose $|a| \ge c x^l\sigma^{-m}$ on $K$. Let $A'_0 = \Op(\chi a^{-1})$, where $\chi \in C^\infty({}^{\textrm{sc}}\overline{T}^*X)$, $\chi \equiv 1$ on $K$, and $|a| \ge \frac c 2 x^l\sigma^{-m}$ on $\supp \chi$. Now the principal symbol of $A'_0A - \Id$ vanishes on $K$, so we have $A'_0A - \Id = R_0$, where $BR_0 \in \Psi^{-1,1,1}(X)$ for any $B\in\Psi^{0,0,0}(X)$ with $\WF_hB \subset K$. Let $r_0$ be the principal symbol of $R_0$. Then put $A'_1 = -\Op(\chi r_0 a^{-1})$. Now $B(A'_0 + A'_1)A - \Id \in \Psi^{-2,2,2}(X)$ for any $B \in \Psi^{0,0,0}$ with $\WF_hB \subset K$. An iteration of this procedure followed by a Borel asymptotic summation gives us
\[ \overline{A}' \sim A'_0 + A'_1 + \cdots \]
with $\overline{A}'\in\Psi^{-m,-l,0}(X)$ satisfying the first half of \eqref{ellip}. Similarly we may produce $\wt{A}'\in\Psi^{-m,-l,0}(X)$ satisfying the second half of \eqref{ellip}. But
\[\overline{A}' - \wt{A}' = \overline{A}'A(\overline{A}' - \wt{A}')A\wt{A}' + \Oh_K(h^\infty) = \Oh_K(h^\infty),\]
where $\Oh_K(h^\infty)$ denotes a psuedodifferential operator whose waverfront set does not intersect $K$. Hence we may arrange $A' = \overline{A}' = \wt{A}'$, and we have achieved \eqref{ellip}.

We will also define the semiclassical wavefront set for a function $u \in C^{\infty}(X^\circ)$ which is \textit{h-tempered}, namely which satisfies $\|x^Nu\|_{L^2(X)} \le Ch^{-N}$ for some $N \in \N$. We say that a point $\zeta \in {}^{\textrm{sc}}T^*X$ is in the complement of $\WF_hu$ if there exist $m,l \in \R$ and $A_0 \in \Psi^{m,l,0}(X)$ such that $A_0$ is elliptic at $\zeta$ and
\begin{equation}\label{wf}\|A_0u\|_{L^2(X)} = \Oh(h^\infty). \end{equation}
In analogy to \eqref{comp} we have, for any $A \in \Psi^{m,l,0}(X)$,
\begin{equation}\label{disj} \WF_h(Au) \subset \WF_hA \cap \WF_hu.\end{equation}
Indeed, if $\zeta \not\in\WF_hA$, then we may take $A_0$ with $\WF_hA_0 \cap \WF_hA = \varnothing$, so that $\WF_h(A_0A) = \varnothing$, and such an operator is $\Oh_{L^2(X) \to L^2(X)}(h^\infty)$ by definition. If, on the other hand, $\zeta \not\in \WF_hu$, then we take $A_0$ as in \eqref{wf}. By ellipticity, there exists $B \in \Psi^{-m,-l,0}(X)$ such that $I = BA_0 + R$ with $\zeta \not\in \WF_h R$. Then $Au = ABA_0u + ARu$. The first term is $\Oh(h^\infty)$, and the second has $\zeta \not\in \WF_h(ARu)$ because $\WF_h(ARu) \subset \WF_h(AR) \subset \WF_h(R)$.

Similarly, if $\zeta \not\in \WF_hu$ and if $A \in \Psi^{m,l,0}(X)$ has $\WF_hA$ contained in a sufficiently small neighborhood of $\zeta$, then
\begin{equation}\label{wflocal} \|Au\|_{L^2(X)} = \Oh(h^\infty). \end{equation}
Indeed, again consider $Au = ABA_0u + ARu$. The first term is already $\Oh(h^\infty)$, and the second will be provided $\WF_h(A) \cap \WF_h(R) = \varnothing$.

Finally
\begin{equation}\label{empty} \WF_hu = \varnothing \quad\Longrightarrow\quad \|x^{-N}u\|_{L^2(X)} = \Oh(h^\infty),\quad\forall N\in\N.\end{equation}
This can be shown by using \eqref{wflocal} and a partition of unity to construct a globally elliptic operator $A$ such that $\|Au\|_{L^2(X)} = \Oh(h^\infty)$.

We emphasize that when $u$ depends on a parameter $\beta$, the implicit constants in \eqref{wf}, \eqref{wflocal} and \eqref{empty} are uniform in $\beta$.

\section{An incoming resolvent estimate} \label{incoming}

We prove here a lemma concerning solutions to the equation $(P - \lambda - i\beta)u = f$, where the principal symbol of $P$ has nontrapping flow at $\lambda$.
We claim that $\WF_h u$ is contained in the forward-in-time bicharacteristics originating in $\WF_h f$. The proof is based on the construction and estimates of \cite{vz}.

\begin{lem} \label{inc} Let $P$ be a self-adjoint operator in $\Psi^{2,0,0}(X)$ whose principal symbol is real and has nontrapping Hamiltonian flow at energy $\lambda$, and suppose $P = -h^2 \Delta_g$ outside of a compact subset of $X^\circ$. Let $f \in C_0^\infty(X^\circ)$ with $\|f\|_{L^2(X)} = 1$, and suppose $u$ solves
\[(P - \lambda - i\beta)u = f.\]
Let $p$ be the principal symbol of $P$. Then, for $h$ sufficiently small and for all $\beta>0$, $\WF_h u \cap T^*X^\circ$ is contained in
\[\left(\bigcup_{t \ge 0} \exp(tH_p)(\WF_h f) \right)\cap \left(p^{-1}(\lambda) \cup \WF_h f\right)\]
\end{lem}

\begin{proof} We proceed in four steps:

\textbf{Step 1.}
We observe first that we can use ellipticity to restrict ourselves to $p^{-1}(\lambda) \cup \WF_h f$. Indeed, suppose $a \in C_0^\infty(T^*X^\circ)$, and suppose that $\supp a \cap \left(p^{-1}(\lambda) \cup \WF_h f\right) = \varnothing$. Using the fact that the principal symbol of $P - \lambda - i\beta$ is nonvanishing on $\supp a$, for $h$ sufficiently small construct a local parametrix $P'$ for $P - \lambda - i \beta$ such that $\supp a \cap (\WF_h(P'(P-\lambda-i\beta) - I)$. Now, using the fact that $\supp a \cap \WF_h f = \varnothing$, we have from \eqref{disj} and \eqref{empty} that
\[\Op(a)u = \Op(a)P'(P - \lambda - i \beta)u + \Oh(h^\infty) = \Oh(h^\infty).\]

\textbf{Step 2.}
Now take $\zeta \in T^*X^\circ \cap p^{-1}(\lambda)$ satisfying $\zeta \not\in \left(\bigcup_{t \ge 0} \exp(tH_p)(\WF_h f) \right)$.
We will need the following fact about the bicharacteristic through $\zeta$: Given any $x_0 > 0$, there exists $T > 0$ such that
\begin{equation}\label{tau} t \le -T \,\Rightarrow\, \tau(\exp(tH_p)(\zeta)) > 2\sqrt \lambda/3,\, x(\exp(tH_p)(\zeta)) < x_0/2,\end{equation}
where $\tau$ comes from the coordinates $(x,y,\tau,\mu)$ near $\D X$. Observe that the conclusion concerning $x(\exp(tH_p)(\zeta))$ follows from the nontrapping hypothesis, so we it suffices to prove the conclusion concerning $\tau(\exp(tH_p)(\zeta))$.

From \eqref{ham}, because our symbol agrees with $|\zeta|^2$ near $\D X$, we have
\[ H_p = 2\tau x^2\D_x + 2\tau x\mu \cdot \D_\mu - (2x h'(\mu,\mu) - x^2 \D_x h'(\mu,\mu))\D_\tau + xH_{h'} \quad \textrm{near } \D X.\]

As in the proof of Lemma \ref{nontrap} we have
\[ x^{-1}(t)\tau(t) \to \infty \textrm{ as } t \to -\infty, \qquad x \textrm{ sufficiently small.}\]
Hence $\tau > 0$, so it remains to show that $|\tau| > 2\sqrt \lambda/3$. Conservation of $p = \tau^2 + h'(\mu,\mu)$ implies that
\[|p - \lambda| < \delta_1, |\tau| \le 2 \sqrt \lambda/3 \Rightarrow h' \ge 2 c_1 > 0.\]
But $h'$ is smooth up to $\D X$, so under these assumptions we have $\frac d {dt} \tau = - 2 x h' + x^2 \D_x h' \le - c_1 x$ for $x$ sufficiently small.
Using $\frac d {dt} x = 2 \tau x^2$, for $|p - \lambda| < \delta_1$ we have
\[\log x(t) = \log x(0) - \int_t^0 2\tau x ds \ge \log x(0) - 2\sqrt{\lambda + \delta_1}\int_t^0 x ds.\]
When $x(0)$ is sufficiently small we thus obtain
\[\tau(t) = \tau(0) - \int_t^0 \frac d {ds} \tau ds \ge \tau(0) + c_1 \int_t^0 x ds \ge \tau(0) + c_1 \frac{\log x(0) - \log x(t)}{2\sqrt{\lambda + \delta_1}}.\]
As $t \to -\infty$, we have $x(t) \to 0$, and hence the right hand side increases without bound. This means that eventually $h' < c_1$, and so $\tau(-t) > 2\sqrt\lambda/3$ and we have \eqref{tau}. 

\textbf{Step 3.}
We will construct a nested family of escape functions which are positive near $\zeta$. More precisely, for $j \in \N$, we construct $q_j \in S^{-\infty,-\eps,0}(X)$, $q_j \ge 0$ everywhere, $\supp q_j \cap \WF_h f = \varnothing$, satisfying:
\[ H_p q_j^2 = -b_j^2, \]
where $b_j \in S^{-\infty,\frac 12 -\eps,0}(X)$, and
\begin{equation}\label{cover} b_1(\overline\zeta) \ge c_1x^{\frac 12 - \eps} \quad \textrm{on} \quad \overline{\bigcup_{t\le0}(\exp(tH_p)(\zeta)}, \qquad b_{j+1} \ge c_j x^{\frac 12 - \eps} \quad \textrm{on} \quad \supp b_j.
\end{equation}

Let $\chi_j \in C^\infty(\R)$ be supported in the interval $(\sqrt \lambda/ 3, \infty)$, identically $1$ on $[2 \sqrt \lambda/ 3, \infty)$, and satisfy $\chi_j' \ge 0$. Suppose further that $\chi_{j+1}\chi_j \equiv \chi_j$. Let $\rho_j \in C_0^\infty ([0,\delta_j))$ be identically $1$ near zero and have $\rho_j'\le 0$, where $\delta_j$ is chosen such that the semi-global coordinates are valid for $x$ in the support of $\rho$, and so that $\rho_j\rho_{j+1} \equiv \rho_j$ while $\inf \delta_j > 0$. Finally take $\psi_j \in C_0^\infty(\R; [0,1]), \psi_j \equiv 1 \textrm{ near } \lambda$, $\supp \psi_j \subset (-\delta + \lambda, \lambda + \delta)$, such that $\psi_{j+1}\psi_j \equiv \psi_j$, and put
\[q_{j,1} = x^{-\eps}\chi_j(\tau) \rho_j(x) \psi_j(p). \]
Now
\begin{align*} H_p q_{j,1} = [-2\eps\tau x^{1-\eps}&\chi_j(\tau) \rho_j(x) + 2x^{2-\eps}\tau \chi_j(\tau) \rho_j'(x) \\
&+ (-2x^{1-\eps}h' + x^{2-\eps}\D_xh'))\chi_j'(\tau) \rho_j(x)]\psi_j(p).\end{align*}
Each term on the right hand side is nonpositive everywhere (for the last term we need to have $\rho_j$ supported in a sufficiently small neighborhood of $0$ to make $|x\D_xh'|$ small), and the first term is negative when $\tau \ge 2 \sqrt \lambda/ 3$, $p \in \psi^{-1}(1)$, $x \in \rho^{-1}(1)$. This $q_{j,1}$ has all the needed properties, except that \eqref{cover} is replaced by
\[ b_1 \ge c_1x^{\frac 12 - \eps} \qquad \textrm{on} \qquad \bigcup_{t\le-T}\exp(tH_p)(\zeta),\qquad b_{j+1} \ge c_j x^{\frac 12 - \eps} \quad \textrm{on} \quad \supp b_j.\]
To complete the construction we put $q_j = q_{j,1} + q_{j,2}$, where $q_{j,2}$ is supported in a tubular neighborhood of $\cup_{-T\le t\le0}\exp(tH_p)(\zeta)$. Indeed, let $U$ be such a tubular neighborhood, taken so small that we can introduce a hypersurface $\Sigma$, transversal to $\cup_{-T\le t\le0}\exp(tH_p)(\zeta)$, such that
\[U = \bigcup_{-T-1\le t \le 1} \exp(tH_p)(U \cap \Sigma).\]
Now let $\phi_j \in C_0^\infty(U\cap\Sigma)$ be identically 1 near $\zeta$ and such that $\phi_j\phi_{j+1} \equiv \phi_j$, and let $\wt\chi_j \in C_0^\infty((-T-1,1))$ satisfy $\wt\chi_j \ge 0$, $\wt\chi_1' < c$ on $[-T,0]$, $\wt\chi_{j+1}' < c$ on $\supp \wt\chi_j$. Now putting $q_2 = \eps_j \phi_j\wt\chi_j\psi(p)$ for $\eps_j$ small enough completes the construction.

\textbf{Step 4}.
The remaining part of the proof is a positive commutator argument, which is the semiclassical adaptation of the proof of \cite[Proposition 3.4.5]{h}. We take $Q_j = \Op(q_j)$, $B_j = \frac 12(\Op(b_j) + \Op(b_j)^*)$, and observe that $H_pq_j^2 = -b_j^2$ implies that
\[ B_j^2  = \frac i h[Q_j^*Q_j,P]+ hx^{1 - \frac\eps2}R_jx^{1-\frac\eps2}, \]
where $R_j \in \Psi^{-\infty,0,0}(X)$. The property \eqref{cover} allows us to construct $A_j \in \Psi^{0,0,0}(X)$ such that $\WF_h(A_j - \Id) \subset \WF_h(B_j)$, while $B_{j+1}$ is elliptic on $\WF_hA_j$. Now, for $\beta > 0$, we have
\begin{align*}
&\|B_ju\|_{L^2(X)}^2 = \la A_ju, B_j^2 A_ju \ra + \Oh(h^\infty) \\
&= \frac ih\la A_ju, [Q_j^*Q_j,P] A_ju \ra + h\la A_ju,x^{1+\frac\eps2}R_jx^{1+\frac\eps2}A_ju\ra + \Oh(h^\infty)\\
&= \frac {-2i}h\left(\im \la u, Q_j^*Q_j(P-\lambda - i\beta)u\ra + \beta\|Q_j A_ju\|_{L^2(X)}^2\right) + h\la A_ju,x^{1+\frac\eps2}R_jx^{1+\frac\eps2}A_ju\ra + \Oh(h^\infty)\\
&\le  Ch\|x^{1+\frac\eps2}A_ju\|_{L^2(X)} + \Oh(h^\infty).
\end{align*}
For the first equality we used $\WF_h B_j \cap \WF_h(A_j-\Id)= \varnothing$. For the inequality we used $\beta\|Q_j A_ju\|_{L^2(X)}^2\ge 0$, $\WF Q_j \cap \WF_h(A_j-\Id)= \varnothing$, and $\WF_h Q_j \cap \WF_h (P-\lambda - i\beta)u = \varnothing$. From \cite[(1.1)]{vz} we know that $x^{\frac 12 + \eps}u \in L^2(X)$ uniformly in $\beta$, so the constants on the right hand side of the inequality are uniform in $\beta$. Next we observe that $B_{j+1}$ is elliptic near $\WF_hA_j$, so we may construct a parametrix, $B'_{j+1} \in \Psi^{-\infty,-\frac 12 + \eps,0}(X)$, such that $\WF_h(B'_{j+1}B_{j+1} - \Id)\cap\WF_hA_j = \varnothing$. This allows us to write
\begin{align}\label{iter}
\|x^{1+\frac\eps2}A_ju\|^2_{L^2(X)} &= \|x^{1+\frac \eps 2}A_jB'_{j+1}B_{j+1}u\|^2_{L^2(X)} + \Oh(h^\infty) \le C\|B_{j+1}u\|^2_{L^2(X)} + \Oh(h^\infty) \notag\\
&\le Ch\|x^{1+\frac\eps2}A_{j+1}u\|^2_{L^2(X)} + \Oh(h^\infty).
\end{align}
We have used the fact that $x^{1+\frac \eps 2}A_jB'_{j+1} \in \Psi^{-\infty,\frac 12 + \frac 32 \eps,0}(X)$ is bounded on $L^2(X)$.

Since \eqref{iter} holds for all $j \in \N$, we find that $\|x^{1+\frac\eps2}A_ju\|_{L^2(X)}^2 = \Oh(h^\infty)$, and since the $x^{1+\frac\eps2}A_j$ are elliptic at $\zeta$ this concludes the proof.
\end{proof}

\section{A preliminary global resolvent estimate}\label{fi}

Put $P = -h^2\Delta$. As a first step we show that
\begin{equation}\label{log2}
\|x^{\frac 12 + \eps}(P- \lambda -i\beta)^{-1}x^{\frac 12 + \eps}\|_{L^2(X) \to H_h^2(X)} \le C \frac {\log^2(1/h)}h.
\end{equation}

To prove this result we will need some auxiliary smooth cutoff functions on $X$. Let $W \in C^\infty(X;[0,1])$ satisfy $W \equiv 1$ in a neighborhood of $\D X$, and for $j \in \{1,2,3\}$, let $\chi_j \in C^\infty(X;[0,1])$ satisfy $\chi_j\chi_{j+1} \equiv \chi_j$ and $\chi_3W \equiv 0$. Suppose further that $\supp(1 - \chi_1)$ is contained in the collar neighborhood of the boundary where we have `semi-global coordinates' $(x,y,\tau,\mu) = (x,y,x^2\xi,x\eta)$ on $^{\textrm{sc}}T^*X$, and that $\chi_a \equiv 1$ on $\pi(K)$, the projection of the trapped set onto $X$. Now from \cite[Proposition 9.2]{nz} we have
\begin{equation}\label{nzest} (P - iW - \lambda - i \beta)u = f \Longrightarrow \|u\|_{H^2_h(X)} \le C \frac {\log(1/h)}h \|f\|_{L^2(X)}. \end{equation}
Further, from \cite[(1.5)]{cv}, we have, for $j \in \{1,2,3\}$
\begin{equation}\label{cvest} (P - \lambda - i\beta)u = (1-\chi_j)f \Longrightarrow \|x^{\frac 12 + \eps}(1-\chi_j)u\|_{H^2_h(X)} \le C \frac 1 h \|x^{-\frac 12 - \eps}f\|_{L^2(X)}.\end{equation}
That the hypotheses of \cite{cv} are satisfied is guaranteed by the normal form of \cite[Proposition 2.1]{js}. As stated in \cite{cv}, the estimate is valid for $\beta$ in an interval smaller than ours, but the stronger statement can be deduced from the weaker one using the resolvent identity.

Take $f \in C^\infty(X^\circ)$ such that $x^{-\frac 12 - \eps}f \in L^2(X)$, and consider $u$ which solves $(P-\lambda - i\beta)u = f$. Our goal is to estimate this $u$, and to do so we will write it as a sum of three functions \eqref{u} which we will estimate individually.
First take $u_0$ such that $(P - iW - \lambda - i\beta)u_0 = \chi_1 f$. We have
\[(P - \lambda - i\beta) \chi_2 u_0 = \chi_2 (P - iW - \lambda - i\beta) u_0 + [P,\chi_2]u_0 = \chi_1 f + [P,\chi_2]u_0.\]
If $(P - \lambda - i\beta)v = (1-\chi_1) f$ and $(P - \lambda - i\beta)u_1 = [P,\chi_2]u_0$, then
\begin{equation}\label{u}u = \chi_2 u_0 + v - u_1.\end{equation}
By \eqref{nzest} we have
\begin{equation}\label{u0} \|\chi_2 u_0\|_{H^2_h(X)} \le C \frac {\log(1/h)}h \|\chi_1 f\|_{L^2(X)}.\end{equation}
By \eqref{cvest} we have
\begin{equation}\label{v1} \|x^{\frac12 + \eps}(1-\chi_1)v\|_{H^2_h(X)} \le C \frac 1 h \|x^{-\frac12 - \eps}(1-\chi_1)f\|_{L^2(X)} \le C \frac 1 h \|x^{-\frac 12 - \eps}f\|_{L^2(X)}.\end{equation}
On the other hand
\[(P - iW - \lambda - i\beta) \chi_2 v = (P - \lambda - i\beta)\chi_2 v = \chi_1 f + [P,\chi_2]v.\]
Now by \eqref{nzest}
\begin{equation}\label{v2}\|\chi_2 v\|_{H^2_h(X)} \le C\frac{\log(1/h)}h(\|\chi_1 f\|_{L^2(X)} + \|[P,\chi_2]v\|_{L^2(X)})\end{equation}
But by \eqref{cvest}
\begin{equation}\label{v3}\|[P,\chi_2]v\|_{L^2(X)} = \|[P,\chi_2](1-\chi_1) v\|_{L^2(X)} \le Ch\|x^{\frac 12 + \eps}(1-\chi_1) v\|_{H^1_h(X)} \le C\|x^{-\frac12 - \eps}f\|_{L^2(X)}.\end{equation}
Plugging \eqref{v3} into \eqref{v2} and combining with \eqref{v1} gives
\begin{equation}\label{v}\|x^{\frac 12 +\eps} v\|_{H^2_h} \le C \frac {\log(1/h)}h \|x^{-\frac 12 - \eps} f\|_{L^2}. \end{equation}
Finally observe that
\[(P - \lambda - i\beta)u_1 = [P,\chi_2]u_0 = (1-\chi_1)[P,\chi_2]\chi_3 u_0,\]
so by \eqref{v}, and \eqref{u0} (the last is applicable because $\chi_3$, like $\chi_2$ has $\chi_3\chi_1 \equiv \chi_1$ and $\chi_3 W \equiv 0$),
\begin{align}\label{u1}\|x^{\frac12 + \eps}u_1\|_{H^2_h(X)} &\le C \frac {\log(1/h)}h \|[P,\chi_2]\chi_3 u_0\|_{L^2(X)} \le C \log(1/h) \|\chi_3 u_0\|_{H^1_h(X)} \notag\\
&\le C \frac {\log^2(1/h)}h \|\chi_1 f\|_{L^2(X)}.\end{align}
Plugging \eqref{u1}, \eqref{v} and \eqref{u0} into \eqref{u} gives
\begin{equation}\label{goal'} (P - \lambda - i\beta)u = f \Longrightarrow \|x^{\frac12 + \eps} u\|_{H^2_h(X)} \le C \frac {\log^2(1/h)} h \|x^{-\frac12 - \eps}f\|_{L^2(X)},\end{equation}
which is the same as \eqref{log2}.

\section{Proof of the theorem}\label{im}

To prove the theorem, we use \eqref{log2} to prove
\begin{equation}\label{two1}\|x^{\frac12 + \eps} u_1\|_{H^2_h(X)} \le C \|\chi_3 u_0\|_{H^1(X)},\end{equation}
improving \eqref{u1}. Then \eqref{u} gives the theorem.

As before we use $(P - \lambda - i\beta)u_1 = [P,\chi_2]u_0 = (1-\chi_1)[P,\chi_2]\chi_3 u_0$ combined with \eqref{cvest} to show that
\[\|x^{\frac12 + \eps} (1-\chi_1)u_1\|_{H^2_h(X)} \le C \frac 1 h \|x^{-\frac12 - \eps}[P,\chi_2]\chi_3 u_0\|_{L^2(X)} \le  C \|\chi_3 u_0\|_{H^1_h(X)}.\]
Hence \eqref{two1} would follow from
\begin{equation}\label{two}\|\chi_1 u_1\|_{H^2_h(X)} \le C \|\chi_3 u_0\|_{H^1_h(X)}.\end{equation}

We begin by taking $\wt P$ to be an operator whose symbol has nontrapping flow at energy $\lambda$, and such that  $(P - \wt P) = (P - \wt P)\chi_1$, and then $\wt u$ such that $(\wt P - \lambda - i\beta)\wt u = [P,\chi_2]u_0$. For example, we may take $\wt P = P + V$, where $V$ is a nonnegative real-valued potential such that $\chi_1 V \equiv V$, but $V > \lambda + 1$ off a small neighborhood of $\D X$ (see Lemma \ref{nontrap} for a proof that this operator is nontrapping near $\D X$). We have immediately from the nontrapping resolvent estimate of \cite[(1.1)]{vz} that
\[\|x^{\frac 12 + \eps}\wt u\|_{H^2_h(X)} \le C \frac 1h \|[P,\chi_2]u_0\|_{L^2(X)} \le C \|\chi_3u_0\|_{H^1_h(X)}\]

Because $u_0$ solves $(P-iW-\lambda-i\beta)u_0 = \chi f$, we know from \cite[Lemma A.2]{nz} that $u_0$ is outgoing i.e. has semiclassical wavefront set contained in the forward flow-out of $\Omega$, where $\Omega$ is the intersection of $T^*_{\supp (d\chi_\circ)}X^\circ$ with the forward flow-out of $\WF_h(\chi f)$. Hence $[P,\chi_2]u_0$ has this property as well, which allows us to deduce from Lemma \ref{inc} that $\pi(\WF_h\wt u) \cap \supp(\chi_1) = \varnothing,$ and hence
\begin{equation}\label{suppo} \pi(\WF_h\wt u) \cap \supp(P-\wt P) = \varnothing. \end{equation}
Now
\[(P - \lambda - i\beta) \wt u = (P-\wt P)\wt u +[P,\chi_2]u_0,\]
so we have
\[u_1 = \wt u + \wt u_1,\]
where $(P - \lambda - i\beta)\wt u_1 = (P-\wt P)\wt u$. Now by \eqref{suppo}, combined with \eqref{disj} and \eqref{empty}, $(P-\wt P)\wt u$ has empty wavefront set and hence is bounded by $\Oh(h^\infty)\|\chi_2 \wt u\|_{L^2(X)}$. Using \eqref{log2}, we conclude the same bound for $\wt u_1$. Hence we have \eqref{two}.

\section{Local smoothing}\label{so}

We now show how the resolvent estimate \eqref{est} gives us local smoothing. This follows an $AA^*$ line of reasoning which we take from \cite[Section 2.3]{bgt} and \cite[p 424]{b2}. The technique used to express the Schr\"odinger propagator in terms of the resolvent is due to Kato \cite[Lemma 3.5]{k}.

We first show how the $L^2 \to L^2$ bound \eqref{theo} implies an $L^2 \to H^2$ bound:
\begin{align*}
\|x^{\frac 12 + \eps} u&\|_{H^2(X)} = \|\Delta_g x^{\frac 12 + \eps}u\|_{L^2(X)} + \|x^{\frac 12 + \eps}u\|_{L^2(X)} \\
&\le \|(-\Delta_g - z - i \beta)x^{\frac 12 + \eps}u\|_{L^2(X)} + (z + \beta)\|x^{\frac 12 + \eps}u\|_{L^2(X)} \\
&\le \|(-\Delta_g - z - i \beta)x^{\frac 12 + \eps}u\|_{L^2(X)} + C z^{\frac 12} \log z \|x^{-\frac 12 - \eps}(-\Delta_g - z - i \beta)u\|_{L^2(X)} \\
&\le C \left(\|[x^{\frac 12 + \eps},\Delta_g]u\|_{L^2(X)} +  z^{\frac 12} \log z \|x^{-\frac 12 - \eps}(-\Delta_g - z - i\beta )u\|_{L^2(X)} \right).
\end{align*}
But$\|[x^{\frac 12 + \eps},\Delta_g]u\|_{L^2(X)} \le C \|x^{\frac 12 + \eps}u\|_{H^1(X)} \le C\nu\|x^{\frac 12 + \eps}u\|_{H^2(X)} + \frac C \nu\|x^{\frac 12 + \eps}u\|_{L^2(X)}$, so we have
\[\|x^{\frac 12 + \eps} u\|_{H^2(X)} \le C z^{\frac 12} \log z \|x^{-\frac 12 - \eps}(-\Delta_g - z -i\beta)u\|_{L^2(X)}.\]
Interpolating between the two bounds using the Riesz-Thorin-Stein Theorem gives
\begin{equation}\label{reso}\left\|x^{\frac 12 + \eps}(-\Delta_g - z \pm i\beta)^{-1} x^{\frac 12 + \eps}\right\|_{L^2(X) \to H^{1 - \eta}(X)} \le C, \qquad \eta>0, \beta \in (0,\beta_0), z \ge z_0.\end{equation}
We observe that the statement about $(-\Delta_g - z + i\beta)^{-1}$ follows from that about $(-\Delta_g - z - i\beta)^{-1}$ by taking the complex conjugate of the estimate.

Now let $A$ be the operator $L^2(X) \to L^2_tH^{\frac 12 - \eta}(X)$ which maps
\[u \mapsto \one_{[0,T]}(t)x^{\frac 12 + \eps} e^{it\Delta_g}u,\]
where $\one_{[0,T]}$ denotes the characteristic function of the interval $[0,T]$, and in our notation we suppress the dependence on the spatial variable. To prove \eqref{smooth}, we must show that $A$ is a bounded operator, or, equivalently, that
\[ AA^*\colon L^2_tH^{-\frac 12 + \eta}(X) \to L^2_tH^{\frac 12 - \eta}(X) \]
is bounded. Observe that $AA^*$ is given by
\begin{align*}AA^* f(t) &=  \one_{[0,T]}(t)x^{\frac 12 + \eps} e^{it\Delta_g}\int_{-\infty}^\infty e^{-is\Delta_g} x^{\frac 12 + \eps}  \one_{[0,T]}(s)f(s) ds.\end{align*}
However, observing that the integral is actually over $[0,T]$, and writing $\int_0^T = \int_0^t - \int_T^t$, we see that it is sufficient to prove
\[\int_0^T\left\|\int_{t_o}^t x^{\frac 12 + \eps} e^{i(t-s)\Delta_g} x^{\frac 12 + \eps}f(s) ds \right\|^2_{H^{\frac 12 - \eta}(X)}dt \le C \int_0^T \|f(t)\|^2_{H^{-\frac 12 + \eta}(X)}dt,\]
where $t_0 \in \{0,T\}$. We put $u_{t_o}(t) = \int_{t_o}^t x^{\frac 12 + \eps} e^{i(t-s)\Delta_g} x^{\frac 12 + \eps}f(s) ds$, and observe that without loss of generality we may assume $\supp f(t) \subset [0,T]$. Observe that as a result we have $\supp u_0(t) \subset [0,\infty)$, and $\supp u_T(t) \subset (-\infty,T]$. This allows us to insert factors of $e^{\pm\beta t}$ into both sides of the estimate to be proven, giving
\[\int_{-\infty}^\infty \left\|e^{-\beta t} u_0(t)\right\|^2_{H^{\frac 12 - \eta}(X)}dt \le C \int_{-\infty}^\infty \|e^{-\beta t}f(t)\|^2_{H^{-\frac 12 + \eta}(X)}dt\]
\[\int_{-\infty}^\infty \left\|e^{\beta t} u_T(t)\right\|^2_{H^{\frac 12 - \eta}(X)}dt \le C \int_{-\infty}^\infty \|e^{\beta t}f(t)\|^2_{H^{-\frac 12 + \eta}(X)}dt.\]
We use Plancherel's theorem to reformulate the two inequalities:
\[\int_{-\infty}^\infty \left\|\hat u_0(z + i \beta) \right\|^2_{H^{\frac 12 - \eta}(X)}dz \le C \int_{-\infty}^\infty \|\hat f(z + i\beta)\|^2_{H^{-\frac 12 + \eta}(X)}dz\]
\[\int_{-\infty}^\infty \left\|\hat u_T(z - i \beta) \right\|^2_{H^{\frac 12 - \eta}(X)}dz \le C \int_{-\infty}^\infty \|\hat f(z - i\beta)\|^2_{H^{-\frac 12 + \eta}(X)}dz.\]
We will prove these pointwise for each $z$: we observe that the functions $u_{t_o}(t)$ solve
\[i \D_t x^{-\frac 12 - \eps}u_{t_o}(t) + \Delta_gx^{-\frac 12 - \eps} u_{t_o}(t) = ix^{\frac 12 + \eps}f(t),\]
and so
\[\hat u_{t_o}(z \mp i \beta) = -ix^{\frac 12 + \eps}(-\Delta_g - z \pm i \beta)^{-1}x^{\frac 12 + \eps}\hat f(z \pm i \beta).\]
In other words it suffices to show that, uniformly in $z \in \R$ and for a fixed $\beta > 0$, we have
\[ \|x^{\frac 12 + \eps}(-\Delta_g - z \pm i \beta)^{-1}x^{\frac 12 + \eps}\|_{H^{\frac 12 - \eta}(X) \to H^{-\frac 12 + \eta}(X)} \le C. \]
But this follows from \eqref{reso}.

We conclude by remarking that under the additional assumption that the cutoff resolvent $x^{\frac 12 + \eps}(-\Delta_g - z \pm i\beta)^{-1}x^{\frac 12 + \eps}$ is bounded on $L^2(X)$ near $z=0$, the above argument may be repeated with $[0,T]$ replaced by $(-\infty,\infty)$ to give local smoothing for infinite time. In this case one uses a density argument, initially taking $f$ compactly supported in time, and finally taking the limit $\beta \to 0$ to obtain a bound uniform in the support of $f$. The behavior of the resolvent near zero has been studied in the case where the bilinear form $h$ in \eqref{metric} is independent of $x$ by Wang \cite{w}, and in the case where $\D X$ is $\mathbb{S}^{n-1}$ with the round metric by Guillarmou-Hassell \cite{gh}.

\textbf{Acknowledgments.} I would like to thank Maciej Zworski for suggesting this problem and for his generous help and guidance with this paper. Thanks also to Colin Guillarmou, Andr\'{a}s Vasy, Dean Baskin, Hans Christianson and Andrew Hassell for helpful discussions. I would particularly like to thank the anonymous referee for two very useful reports. Finally, I am grateful for support from NSF grant DMS-0654436 and from a Phoebe Hearst fellowship.

\end{document}